\documentclass[11pt,reqno]{amsart}
 \pdfoutput=1
\usepackage{amssymb,latexsym}
\usepackage{fancyhdr,amssymb}

\theoremstyle{plain}
\numberwithin{equation}{section}
\newtheorem{thm}{Theorem}[section]
\newtheorem{theorem}[thm]{Theorem}

\begin{document}
\fancyhead{}
\renewcommand{\headrulewidth}{0pt}
\fancyfoot{}
\fancyfoot[LE,RO]{\medskip \thepage}
\fancyfoot[LO]{\medskip MONTH YEAR}

\setcounter{page}{1}

\title [ an explicit  formula for the prime counting function]{  an explicit  formula for the prime counting function}
\author{konstantinos gaitanas}
\address{Department of Applied Mathematical and Physical Sciences\\
                National Technical University of Athens\\
                Greece}
\email{raffako@hotmail.com}

\begin{abstract}
This paper studies the behaviour of the prime counting function at some certain points.\\
We show that there is an exact formula for $\pi(n)$ which is valid  for infinitely many naturals numbers $n$.
\end{abstract}

\maketitle

\section{Introduction}
The prime counting function is at the center of mathematical research for centuries and many asymptotic distributions of $\pi(n)$ are well known.\\
Many formulas have been discovered by mathematicians \cite{Hardy} but almost all of them are using all the prime numbers not greater than  $n$ to calculate $\pi(n)$.
In this paper we  give an exact formula which holds  when a standard condition is satisfied.

\section{Some basic theorems}
\begin{theorem}\label{A}  Let $\pi(n)$ be the number of primes not greater than $n$ and $n\geq1$. \\
  Then $\frac{n}{\pi(n)}$ takes on every integer value greater than $1$.

\end{theorem}

\begin{proof}
 The proof  is presented in \cite{Golomb}. It uses only the fact that $\pi(N) = o(N)$ and $\pi(N+1)-\pi(N)$ is $0$ or $1$.  \\
We can conclude from this theorem that $\frac{n}{\pi(n)}$ is an integer infinitelly often and we will use this fact in order to prove the existence of our formula.
\end{proof}

\begin{theorem}\label{A} $\frac{n}{lnn-\frac{1}{2}}<\pi(n)<\frac{n}{lnn-\frac{3}{2}}$ for every $n\geq 67$.

\end{theorem}

\begin{proof}
This is a theorem proved by J. Barkley Rosser and Lowell Schoenfeld and the proof is presented at \cite{Rosser}.\\
\end{proof}

\section{The formula for $\pi(n)$}
\begin{theorem}\label{A} For infinitely many natural numbers $n$ the following formula is valid:\\
 \begin{center}
 $\pi(n)=\frac{n}{\Bigl\lfloor {{lnn-\frac{1}{2}}\Bigr\rfloor}}$
\end{center}
\end{theorem}
\begin{proof}
We will make use of the above mentioned inequality in order to prove our formula. \\   
\\
We have $\frac{n}{lnn-\frac{1}{2}}<\pi(n)<\frac{n}{lnn-\frac{3}{2}}$ for every $n\geq 67$.\\
  Inverse the inequality
 and multiply by $n$. We can see  that the inequality now has the form:\\
   \begin{center}
$lnn-\frac{3}{2}<\frac{n}{\pi(n)}<lnn-\frac{1}{2}$ .
\end{center}
So $\frac{n}{\pi(n)}$ lies between two real numbers $a-1$ and $a$ with $a=lnn-\frac{1}{2}$.\\
\\
This means that for every $n\geq 67$ when $ \frac{n}{\pi(n)}$ is an integer we must have:\\ 
\begin{center}
$\frac{n}{\pi(n)}=\Bigl\lfloor {lnn-\frac{1}{2}}\Bigr\rfloor \Leftrightarrow \pi(n)=\frac{n}{\Bigl\lfloor {lnn-\frac{1}{2}}\Bigr\rfloor}$.
\end{center}
This completes the proof.
\end{proof}

We can see below at Table 1 that the formula $\Bigl\lfloor  \frac{n}{lnn-\frac{1}{2}}\Bigr\rfloor$ gives exactly the value of $\pi(n)$ for every natural number $67\leq n<4000$  with $\frac{n}{\pi(n)}$ being an integer\cite{sloane}.\\

\begin{center}
\begin{tabular}{|c|c|c|}
\hline $n$ & $\pi(n)$ & $\Bigl\lfloor  \frac{n}{lnn-\frac{1}{2}}\Bigr\rfloor$  \\ \hline \hline $96$ & $24$
& $24$\\ \hline
$100$ & $25$ & $25$ 
 \\ \hline $120$ &
$30$ & $30$ 
 \\ \hline
$330$ & $66$ & $66$
 \\ \hline  $335$&$67$&$67$\\ \hline  $340$&$68$&$68$\\ \hline $350$&$70$&$70$\\ \hline $355$&$71$&$71$\\ \hline $360$&$72$&$72$\\ \hline $1008$&$168$&$168$\\  \hline$1080$&$180$&$180$\\  \hline$1092$&$182$&$182$\\ \hline $1116$&$186$&$186$\\ \hline $1122$&$187$&$187$\\ \hline $1128$&$188$&$188$\\  \hline$1134$&$189$&$189$\\ \hline $3059$&$437$&$437$\\ \hline $3066$&$438$&$438$\\ \hline $3073$&$439$&$439$\\  \hline$3080$&$440$&$440$\\ \hline $3087$&$441$&$441$\\  \hline$3094$&$442$&$442$\\ \hline
\end{tabular}
\end{center}

\begin{center}
TABLE 1
\end{center}

\medskip

\noindent MSC2010: 11A41, 11A25

\end{document}